\documentclass[12pt,reqno]{amsart}
\usepackage[headings]{fullpage}
\usepackage{amssymb,amsmath,mathtools,bbm,tikz,tikz-cd,color,relsize,multirow}
\usepackage[all,cmtip]{xy}
\usepackage{url}
\usepackage[group-separator={,},group-minimum-digits={3}]{siunitx}
\usetikzlibrary{positioning}
\usetikzlibrary{calc}
\usetikzlibrary{decorations.markings}
\usetikzlibrary{arrows}
\usetikzlibrary{calc}
\usetikzlibrary{decorations.markings}
\tikzstyle{hvector}=[inner sep=2pt,draw=blue!50,fill=blue!10,thick]
\tikzstyle{unit}=[inner sep=2pt,shape=circle, draw]
\tikzstyle{counit}=[inner sep=2pt,shape=circle, draw,fill=gray]
\tikzstyle{antipode}=[inner sep=2pt,shape=rectangle, draw]
\tikzstyle{cocycle}=[inner sep=2pt,shape=circle, draw]
\tikzstyle{twistedm}=[inner sep=2pt,shape=circle, fill=gray]
\tikzstyle{autom}=[inner sep=2pt,shape=circle, draw]
\tikzstyle{coact}=[inner sep=2pt,shape=circle, fill=black]

\usepackage[bookmarks=true,%
    colorlinks=true,%
    linkcolor=blue,%
    citecolor=blue,%
    filecolor=blue,%
    menucolor=blue,%
    urlcolor=blue,%
    breaklinks=true]{hyperref}
\usepackage{slashed}    
\usepackage{listings}        
\usepackage{verbatim}
\usepackage[normalem]{ulem}  
\usepackage{diagbox}

\newtheorem{theorem}{Theorem}[section]
\theoremstyle{definition}

\newtheorem{lemma}[theorem]{Lemma}

\newtheorem{remark}[theorem]{Remark}

\newtheorem{question}[theorem]{Question}

\def\BZ{\mathbbm Z}
\def\BQ{\mathbbm Q}

\def\BC{\mathbbm C}

\def\calB{\mathcal B}

\def\ve{\varepsilon}

\def\be{\begin{equation}}
\def\ee{\end{equation}}

\def\vphi{\varphi}

\def\End{\mathrm{End}}

\def\FK{\mathrm{FK}}
\def\FKP{\mathrm{FK3}}
\def\YD{\mathcal{YD}}

\definecolor{codegreen}{rgb}{0,0.6,0}
\definecolor{codegray}{rgb}{0.5,0.5,0.5}
\definecolor{codepurple}{rgb}{0.58,0,0.82}
\definecolor{backcolour}{rgb}{0.95,0.95,0.92}

\lstdefinelanguage{PARIGP}{
  keywords={typeof, new, true, false, catch, function, return, null, catch,
    switch, var, if, in, while, do, else, case, break},
  keywordstyle=\color{blue}\bfseries,
  ndkeywords={class, export, boolean, throw, implements, import, this},
  ndkeywordstyle=\color{darkgray}\bfseries,
  identifierstyle=\color{black},
  sensitive=false,
  comment=[l]{//},
  morecomment=[s]{/*}{*/},
  commentstyle=\color{purple}\ttfamily,
  stringstyle=\color{red}\ttfamily,
  morestring=[b]',
  morestring=[b]"
}

\lstdefinestyle{code}{
    backgroundcolor=\color{backcolour},   
    commentstyle=\color{codegreen},
    keywordstyle=\color{magenta},
    numberstyle=\tiny\color{codegray},
    stringstyle=\color{codepurple},
    basicstyle=\ttfamily\footnotesize,
    breakatwhitespace=false,         
    breaklines=true,                 
    captionpos=b,                    
    keepspaces=true,                 
    numbers=left,                    
    numbersep=5pt,                  
    showspaces=false,                
    showstringspaces=false,
    showtabs=false,                  
    tabsize=2
}

\begin{document}

\title[The $\FK_3$ knot polynomial]{The $\FK_3$ knot polynomial}
\author{Stavros Garoufalidis}
\address[Stavros Garoufalidis]{
  International Center for Mathematics, Department of Mathematics \\
  Southern University of Science and Technology \\
  Shenzhen, China \newline
  {\tt \url{http://people.mpim-bonn.mpg.de/stavros}}}
\email{stavros@mpim-bonn.mpg.de}

\author{Shana Yunsheng Li}
\address[Shana Yunsheng Li]{
  Department of Mathematics \\
  University of Illinois \\
  Urbana, IL, USA \newline
  {\tt \url{https://shana-y-li.github.io}}}
\email{yl202@illinois.edu}

\thanks{
  {\em Key words and phrases:}
  Nichols algebras with automorphisms, Yetter--Drinfel'd modules, $R$-matrix,
  Yang--Baxter equation, knots, knot polynomials, sporadic Nichols algebras,
  $\FK_3$, Fox colorings, Fox 3-colorings.
}

\date{3 September 2026}
\dedicatory{
  To Nicol\'{a}s Andruskiewitsch, with admiration}

\begin{abstract}
  We present the knot polynomial associated to the 12-dimensional sporadic
  Nichols algebra $\FK_3$ with automorphism, compute it for all knots with $\leq$
  16 crossings and discuss the found patterns. This knot polynomial is
  not a Vassiliev power series. We ponder how it compares to the knot polynomials
  that come from Lie algebras and their representations and what the other
  knot polynomials associated with the higher dimensional sporadic Nichols
  algebras are.
\end{abstract}

\maketitle

{\footnotesize
\tableofcontents
}



\section{Introduction}
\label{sec.intro}

It was recently discovered by Kashaev and one of the authors that a finite dimensional
Nichols algebra $H$ with automorphism gives an explicit rigid solution to the
Yang--Baxter equation~\cite{GK:multi}, and hence, by the Reshetikhin--Turaev
functor~\cite{RT:ribbon}, gives rise to multivariable polynomials of knots in
3-space.

In this note we study the knot polynomial obtained by the sporadic Nichols algebra
which is the smallest in the sense of dimension, namely the 12-dimensional $\FK_3$
Nichols algebra. The corresponding knot polynomial lies in $\BZ[t^{\pm 1}]$.
After reviewing the basic properties of $\FK_3$ in Sections~\ref{sub.sporadic}
~and~\ref{sec.S3YD}, we give the detailed definition and computation of the
corresponding $R$-matrix in Section~\ref{sec.FK3poly}, along with the computation of
the knot polynomial and its found patterns. Unlike the knot polynomials that come
from the representation theory of (quantized) simple Lie algebras and Lie superalgebras,
this knot polynomial is not a Vassiliev power series invariant and does not satisfy
a version of the Melvin--Morton--Rozansky Conjecture. We summarize its properties in
Section~\ref{sec.further}.

\section{A sporadic Nichols algebra}
\label{sub.sporadic}

In this section we review the definition and properties of
our hero, a 12-dimensional sporadic Nichols algebra often 
denoted by $\FK_3$ after Fomin--Kirillov. It belongs to a family $\FK_n$
of quadratic algebras introduced by Fomin--Kirillov~\cite{Fomin-Kirillov}.
Curiously, it turns out that for $n=3,4,5$, these quadratic algebras
have two further independent properties: they are finite dimensional of dimension
12, 576 and 8294400 and also, they are Nichols algebras of rank 3, 6, 10.
On the other hand, for $n>5$, $\FK_n$ is neither known to be finite dimensional
nor Nichols. For an excellent survey of Nichols algebras, see~\cite{A:on.fin.dim}.

The Nichols algebras $\FK_n$ for $n=3,4,5$ are non-diagonally braided and
described in terms of a quandle for the conjugacy class $(12)$ of the symmetric group
$S_n$; see~\cite{Grana:web} and also García--García Iglesias~\cite{GI} for $n=4,5$.

The fact that the quadratic algebra $\FK_3$ is a Nichols algebra was shown by
by Milinski--Schneider around 1996 and published 4 years later~\cite{Milinski}
and independently by Gr\~ana; see~\cite{Grana:web}.

In this paper we will focus
exclusively on $n=3$ and denote the braided Hopf algebra $\FK_3=\mathfrak{B}(V)$.
This algebra has no deformations, and it is sporadic in the sense of not fitting
into an infinite family of Nichols algebras. The algebra appears in a small list
of elementary finite dimensional braided Hopf algebras given in~\cite[Tab.1]{Grana}
and also in~\cite[Table 9.1]{Heckenberger:cubic}.

\noindent $\bullet$ {\bf Basis.}
$V$ has basis $\{a,b,c\}$. $\FK_3$ is a quotient of the free associative algebra
over $\BQ$ with generators $a$, $b$ and $c$, modulo the 2-sided ideal of quadratic
relations generated by
\be
\label{ideal2}
\{aa, bb, cc, ab + bc + ca, ba + cb + ac \} \,. 
\ee
$\FK_3$ does not have a PBW basis. A basis of $\FK_3$ as a $\BQ$-vector space
is given by
\be
b=(1, a, b, c, ab, ba, ac, bc, abc, aba, bac, abac) \,.
\ee
We write $b=(b_0,\dots,b_{11})$. 
The Poincare polynomial of $\FK_3$ is $1+3 t + 4 t^2 + 3 t^3 + t^4$.

\noindent $\bullet$ {\bf Multiplication.}
If $\nabla : \FK_3^{\otimes 2}\to \FK_3$ denotes the multiplication, the
multiplication matrix $\nabla(b_i,b_j)$ can be obtained from the above relations
and it is given by the matrix
\begin{tiny}
\be
\label{nablamat}
\left(
\begin{array}{cccccccccccc}
 b_0 & b_1 & b_2 & b_3 & b_4 & b_5 & b_6 & b_7 & b_8 & b_9 & b_{10} & b_{11} \\
 b_1 & 0 & b_4 & b_6 & 0 & b_9 & 0 & b_8 & 0 & 0 & b_{11} & 0 \\
 b_2 & b_5 & 0 & b_7 & b_9 & 0 & b_{10} & 0 & b_{11} & 0 & 0 & 0 \\
 b_3 & -b_4-b_7 & -b_5-b_6 & 0 & b_{10} & b_8 & -b_8 & -b_{10} & 0 & -b_{11} & 0 & 0 \\
 b_4 & b_9 & 0 & b_8 & 0 & 0 & b_{11} & 0 & 0 & 0 & 0 & 0 \\
 b_5 & 0 & b_9 & b_{10} & 0 & 0 & 0 & b_{11} & 0 & 0 & 0 & 0 \\
 b_6 & -b_8 & -b_9 & 0 & b_{11} & 0 & 0 & -b_{11} & 0 & 0 & 0 & 0 \\
 b_7 & -b_9 & -b_{10} & 0 & 0 & b_{11} & -b_{11} & 0 & 0 & 0 & 0 & 0 \\
 b_8 & 0 & -b_{11} & 0 & 0 & 0 & 0 & 0 & 0 & 0 & 0 & 0 \\
 b_9 & 0 & 0 & b_{11} & 0 & 0 & 0 & 0 & 0 & 0 & 0 & 0 \\
 b_{10} & -b_{11} & 0 & 0 & 0 & 0 & 0 & 0 & 0 & 0 & 0 & 0 \\
 b_{11} & 0 & 0 & 0 & 0 & 0 & 0 & 0 & 0 & 0 & 0 & 0 \\
\end{array}
\right) \,.
\ee
\end{tiny}

\noindent $\bullet$ {\bf $S_3$ Yetter--Drinfeld structure.}
$V$ is a $\BQ[S_3]$ Yetter--Drinfeld-module, hence a graded
$S_3$-vector space where the degree of $a$, $b$ and $c$ is
$(12)$, $(13)$ and $(23)$. In fact, $V$ is the YD-module
$V=M([(12)], \text{sgn})$ where $[(12)]=\{(12),(13),(23)\}$ is the conjugacy
class of $(12)$ on $S_3$ with stabilizer of $(12)$ the subgroup
$S_{3,(12)}=\{1, (12)\}$ and $\text{sgn}:S_{3,(12)} \to \mathrm{GL(\BC)}$
the sign representation. For a discussion of Yetter--Drinfeld modules, see for
example~\cite[Sec.1.1]{AG:racks} and for a description of the quandle,
see~\cite{Grana:web}. The action of the $\FK_3$ quandle is $x \triangleright y =
y^x = xyx^{-1}$ for $x,y \in [(12)]$. 

Note that $V$, and hence the Nichols algebra $\calB(V)$ is a YD $\BQ[S_3]$-module,
in particular the multiplication, co-multiplication, antipode and automorphism
are maps of $\BQ[S_3]$ YD-modules.

\noindent $\bullet$ {\bf Braiding.}
$\FK_3$ does not have diagonal braiding. Instead it is given by
$\tau(x \otimes y) = y^x \otimes x$ for the action of $x$ on $y$ 
(both being words in $a$, $b$ and $c$) which is uniquely determined by
\be
y^{x_1 x_2} = (y^{x_2})^{x_1}, \qquad (y_1 y_2)^x = y_1^x y_2^x
\ee
for all $x$, $y$, $x_1$, $x_2$ words in $a$, $b$, $c$ together with the initial
condition
\be
a^a = -a, \qquad a^b = -c, \qquad a^c = -b 
\ee
plus all permutations of $\{a,b,c\}$. So the braiding is uniquely determined by
its values $\tau(x \otimes y)$ for $x,y \in \{a,b,c\}$ given by the matrix
\be
\label{braidingmat}
\left(
\begin{array}{ccc}
-a \otimes a & -c \otimes a & -b \otimes a \\      
-c \otimes b & -b \otimes b & -a \otimes b \\
-b \otimes c & -a \otimes c & -c \otimes c
\end{array}
\right)    \,.
\ee

\noindent $\bullet$ {\bf Co-multiplication.}
The commultiplication is the algebra map uniquely determined by declaring $a$, $b$
and $c$ to be
primitive:
\be
\label{deltamap}
\Delta(x) = x\otimes 1 + 1 \otimes x, \qquad x=a,b,c \,.
\ee
For example, 
\be
\label{delab}
\Delta(ab) = 1 \otimes ab + a \otimes b - c \otimes a + ab \otimes 1 
\ee
so that
\be
\Delta(b_4) = b_0 \otimes b_4 + b_1 \otimes b_2 - b_3 \otimes b_1 + b_4 \otimes b_0 \,.
\ee

\noindent $\bullet$ {\bf Antipode.}
The antipode is uniquely determined by
\be
\label{antimat}
S(x) = -x, \qquad x \in \{a,b,c\} \,.
\ee
For instance,
\be
S(ab)=-ca = ab + bc, \qquad S(b_4) = b_4 + b_7 \,.
\ee

\noindent $\bullet$ {\bf Automorphism.}
It is determined by
\be
\label{automat}
\phi(x) = t x, \qquad x \in \{a,b,c\} \,.
\ee

We end this section with two remarks, not needed in the paper, but stated for
completeness.

\begin{remark}
\label{rem.Fk3iso}
If we start with the braided 3-dimensional vector space $V$ with braiding given
by~\eqref{braidingmat}, using Equation ~\eqref{delab} three times, we find that
the quadratic generators of the ideal ~\eqref{ideal2} are primitive and hence
vanish in $\calB(V)$. This gives an onto map from the quadratic algebra in three
generators modulo the relations~\eqref{ideal2} (i.e., the original Fomin--Kirillov
algebra) to $\calB(V)$. One can show that this is actually an isomorphism;
see~\cite{Milinski} and~\cite{Grana}. To achieve this, one uses the skew derivations
to show that there are no more primitives.
\end{remark}

\begin{remark}
\label{rem.double}  
There is a close connection between Nichols algebras and quantum groups. Namely,
after bosonization and doubling, diagonally braided Nichols algebras often are
identified with quantum groups; see Heckenberger~\cite{He:drinfeld}. This is not
the case for non-diagonally braided Nichols algebras. In~\cite{PV:1,PV:2}
Pogorelsky--Vay have studied the double of the bosonization of $\FK_3$.
\end{remark}


\section{YD-modules over $S_3$}
\label{sec.S3YD}

\subsection{YD-modules over $G$}

In this section we recall for completeness some well-known facts about
YD-modules of a finite group $G$, and in more detail, for $G=S_3$.  
A detailed exposition is given in ~\cite{Witherspoon}. Although this section
is not needed at all for the definition and computation of the $\FK_3$ knot
polynomial, one expects that the two are somehow related, and that the polynomial
offers a nontrivial $t$-deformation of the counting knot invariants presented here.

All our modules will be finite-dimensional vector spaces, and left with respect
to the Hopf algebra action. We will fix a finite group $G$, and let $\BQ G$ denote
its Hopf algebra and $D(\BQ G) = \BQ G \otimes (\BQ G)^*$ denote its Drinfeld double.
It is a Hopf algebra with a quasi-triangular structure. 

There is an equivalence of categories of left modules over $D(\BQ G)$ and
the category ${}^{\BQ G}_{\BQ G}\YD$ of left YD-modules over $G$
(abbreviated by $G$-YD-modules). The simple modules of the latter category are
of the form $M([g], \rho) = \BQ G \otimes_{\BQ G_g} V $ where $[g]$ is
a conjugacy class in $G$ and $\rho: G_g \to \End(V)$ is an irreducible representation
of the stabilizer $G_g$ of $g$~\cite{Witherspoon}. 

\subsection{Knot invariants count $G$-homomorphisms}

We now discuss the knot invariants introduced by Eisermann~\cite{Eisermann} following
earlier ideas of knot invariants coming from quandles, and early work of
Dijkgraaf--Witten on Chern--Simons theory with a finite gauge group.

The invariant $J^{(M[g],\rho)}_K$ of a knot $K$ associated to $M([g], \rho)$ is 
\be
\label{ZMK}
J^{(M[g],\rho)}_K = \sum_{\substack{\phi \in \text{Hom}(\pi_1(S^3 \setminus K), G)
    \\ [\phi(\mu)] = [g]}}
\chi_{\rho}(\phi(\lambda))
\ee
Here $\mu$ and $\lambda$ are a choice of a meridian and longitude of $K$. $\mu$ and
$\lambda$ commute, and $\lambda$ lies in the commutator subgroup of
$\pi_1(\S^3\setminus K)$. These invariants are normalized at the unknot so that
$J{(M[g],\rho)}_\text{Unknot} = \dim(M[g],\rho)$.  

\subsection{The case of $S_3$}

For $G=S_3$, there are 8 simple modules:

\be
\label{Msimple}
\begin{aligned}
M_1 &= M([1], \ve), \qquad M_2 = M([1], \text{sgn}), \qquad M_3 = M([1], \text{2-dim})
\\
M_4 &= M([(12)], \ve), \qquad M_5 = M([(12)], \text{sgn})
\\
M_6 &= M([(123)], \ve), \qquad M_7 = M([(123)], \zeta_3), \qquad
M_8 = M([(123)], \zeta_3^2) 
\end{aligned}
\ee
where $\zeta_3=e^{2\pi i/3}$ and $\ve$ denotes the trivial representation. They
have (quantum) dimensions 1,1,1,3,3,2,2,2, respectively.
The decomposition of $\calB(V)$ as a $S_3$-YD-module is given by
\begin{small}
\be
\label{BS3}
\calB(V)_0 = M_2, \quad
\calB(V)_1 = M_5, \quad
\calB(V)_2 = M_7 \oplus M_8 , \quad
\calB(V)_3 = M_5, \quad
\calB(V)_4 = M_2 . \quad
\ee
\end{small}

\begin{lemma}
\label{lem.S3}  
The knot invariants corresponding to the above simple modules are given
by
\be
\begin{aligned}
J^{M_1}(K) &= J^{M_2}(K) = J^{M_3}(K) = 1 \\
J^{M_4}(K) &= J^{M_5}(K) = F_3(K) \\
J^{M_6}(K) &= J^{M_7}(K) = J^{M_8}(K) = 2
\end{aligned}
\ee
where $F_3(K)$ is the number of Fox 3-colorings of $K$.
\end{lemma}

\begin{proof}
For $M_1$, $M_2$ and $M_3$, it is obvious. For $M_4$ and $M_5$, it holds because
$[(12)]=\{(12),(13),(23)\}$ and the commutator of any two of them is the third.
Hence a coloring of the Wirtinger over-arcs is either constant, or all 3 colors
appear. Moreover, $\lambda$ is in the commutator subgroup of $\pi_1(S^3\setminus K)$
hence $\phi(\lambda)$ is in $[S_3,S_3]$ and in $S_{3,(12)}=\{1,(12)\}$ hence
$\phi(\lambda)=1$.

Finally, for $M_6$, $M_7$ and $M_8$, it
holds because $[(123)]=\{(123),(132)\}$ and $(123)$ commutes with $(132)$
hence a coloring of the Wirtinger over-arcs is constant. Moreover, $\lambda$
is in the commutator subgroup of $\pi_1(S^3\setminus K)$ hence $\phi(\lambda)=1$.
\end{proof}


\section{The $\FK3$ polynomial of knots}
\label{sec.FK3poly}

\subsection{The $R$-matrix}
\label{sub.R}

We now discuss the $R$-matrix associated to a finite dimensional Nichols algebra
$H$ with commultiplication $\Delta: H \to H^{\otimes 2}$, multiplication
$\nabla:H^{\otimes 2} \to H$, antipode $S: H \to H$, braiding $\tau: H^{\otimes 2} \to
H^{\otimes 2}$ and automorphism $\vphi: H \to H$. The (left) $R$-matrix is the 
linear endomorphism
\be
R \in \End(H^{\otimes 2})
\ee
given explicitly by
\begin{equation}
\label{eq:Rtensor}
R =(\nabla\otimes\operatorname{id}_H)(\operatorname{id}_H\otimes \tau)
(\delta\otimes \varphi), \qquad
\delta =(\nabla\otimes\operatorname{id}_H)(\operatorname{id}_H\otimes \tau)
(\operatorname{id}_{H\otimes H}\otimes S\varphi)\Delta^{(2)} \,.
\end{equation}
Here, $\Delta^{(2)} : H \to H^{\otimes 3}$ denotes the twice iterated comultiplication.
Theorems 3.5 and 3.6 of \cite{GK:multi} show that $R$ is an invertible linear
endomorphism that satisfies the Yang--Baxter equation, and in fact it is a
rigid $R$-matrix.

Using the basis $b_i \otimes b_j$ for $i,j=0,\dots,11$ for $\FK_3^{\otimes 2}$, the
$R$-matrix becomes a $12^2 \times 12^2$ matrix with entries in the structure constants
of the multiplication, commultiplication, antipode and automorphism. We performed
its computation in two stages. Recall that $\FK_3$ is the quotient of the
tensor algebra $T(V)=\BQ\langle a,b,c \rangle$, which is a free associative algebra
in three non-commuting generators $a$, $b$, $c$. In the first stage, we computed the
entries of the $R$-matrix in the free algebra $T(V)$, and in the second we reduced them 
by the quadratic ideal~\eqref{ideal2}. After doing so, we obtained a matrix
with entries in $\BZ[t^{\pm 1}]$. 

As a consistency check, we verified that the $R$-matrix presented below
satisfies the Yang--Baxter equation. $R$ has determinant $t^{288}$. $R$ is a sparse
matrix: only $477$ out of its $20736$ entries (less than $3\%$) are nonzero.
All entries of $R$ are in $\BZ[t^{\pm 1}]$ of degree at most $4$ with integer
coefficients of absolute value at most 3. The characteristic polynomial of $R$
is given by
\be
(-1 + x)^{12} (t + x)^{36} (t^3 + x)^{36} (-t^4 + x)^{12} (t^4 - t^2 x + x^2)^{24} \,.
\ee
The minimal polynomial for $t^{12} \neq 1$ is
\be
(-1 + x) (t + x) (t^3 + x) (-t^4 + x) (t^4 - t^2 x + x^2) \,,
\ee
and for $t=1$ is $(-1 + x) (1 + x)^2 (1 - x + x^2)$.
The $R$-matrix is $S_3$-homogeneous, where the degree of $a$, $b$ and $c$ is
$(12)$, $(13)$ and $(23)$. 

In the appendix we give the $R$-matrix in condensed form as a $12 \times 12$ matrix
$(R(b_i \otimes b_j))_{i,j=0,\dots,11}$. 

\begin{remark}
\label{rem.LR}
Theorems 3.5 and 3.6 of \cite{GK:multi} associate two $R$-matrices (denoted by
$\rho_L$ and $\rho_R$) to the same finite-dimensional Nichols algebra with
automorphism, depending on considering the Nichols algebra as a left or right module.
In the above discussion, we presented only the left $R$-matrix. We have also computed
the right $R$-matrix, but we found out that experimentally, the knot polynomials
associated to each one of them coincide. Without doubt, the two $R$-matrices are
twist-equivalent, or even conjugate by a linear automorphism of the Nichols algebra. 
\end{remark}

\subsection{The knot polynomial}
\label{sub.poly}

Using the above $R$-matrix, one can associate via the Reshetikhin--Turaev
construction~\cite{RT:ribbon} a polynomial invariant to each long
knot $K$, which takes values in $\mathrm{Aut}(H \otimes H \otimes \BZ[t^{\pm 1}])$
~\cite{GK:multi}. Using the basis of $H$, we arrive at a $12 \times 12$ matrix with
coefficients in $\BZ[t^{\pm 1}]$. Let $\FKP_K(t) \in \BZ[t^{\pm 1}]$ denote the
$b_0 \otimes b_0$ entry of this matrix.

The computation of this knot polynomial has been implemented both in
\texttt{Mathematica} and in \texttt{python}, the latter part of the development
version of \texttt{SnapPy} package~\cite{snappy}. The python implementation
is considerably faster, and its usage (assuming \texttt{SnapPy} is available
in \texttt{sage}) is the following:

\begin{lstlisting}
sage: from FK3 import *
sage: FK3(snappy.Link('4_1'))
t^-4 - 9*t^-3 + 33*t^-2 - 69*t^-1 + 89 - 69*t + 33*t^2 - 9*t^3 + t^4
\end{lstlisting}

Using the implementation in \cite{Li:FPT-RT}, in less than one day of parallel
computation using 200 threads, we computed this polynomial for all 1.7 million
knots with $\leq 16$ crossings; the code and the computed data are available
at~\cite{GL:FK3data}. The data revealed several experimental observations which
we now state. We will not attempt to prove these observations, except for those
that follow directly from the exhibited computed data.

\noindent $\bullet$ {\bf First values.}
The matrix-valued invariant of $K$ is $\FKP_K(t)$ times the identity matrix. The
polynomials for knots with $\leq 6$ crossings are:

\begin{tiny}
\be
\label{first}
\begin{aligned}
\FKP_{3_1} &=
t^{-4} - 3 t^{-3} + 7 t^{-2} - 9 t^{-1} + 9 - 9 t + 7 t^2 - 3 t^3 + t^4
\\
\FKP_{4_1} &=
t^{-4} - 9 t^{-3} + 33 t^{-2} - 69 t^{-1} + 89 - 69 t + 33 t^2 - 9 t^3 + t^4
\\
\FKP_{5_1} &=
t^{-8} - 3 t^{-7} + 7 t^{-6} - 12 t^{-5} + 18 t^{-4} - 21 t^{-3} + 20 t^{-2}
- 18 t^{-1} + 17 - 18 t + 20 t^2 - 21 t^3 + 18 t^4 - 12 t^5 + 7 t^6 \\ & - 3 t^7 + t^8
\\
\FKP_{5_2} &=
10 t^{-4} - 48 t^{-3} + 120 t^{-2} - 192 t^{-1} + 221 - 192 t + 120 t^2 - 48 t^3
+ 10 t^4
\\
\FKP_{6_1} &=
10 t^{-4} - 72 t^{-3} + 226 t^{-2} - 426 t^{-1} + 525 - 426 t + 226 t^2 - 72 t^3
+ 10 t^4
\\
\FKP_{6_2} &=
t^{-8} - 9 t^{-7} + 39 t^{-6} - 114 t^{-5} + 246 t^{-4} - 411 t^{-3} + 552 t^{-2}
- 624 t^{-1} + 641 - 624 t + 552 t^2 - 411 t^3 + 246 t^4 \\ & - 114 t^5 + 39 t^6
- 9 t^7 + t^8
\\
\FKP_{6_3} &=
t^{-8} - 9 t^{-7} + 45 t^{-6} - 150 t^{-5} + 370 t^{-4} - 717 t^{-3} + 1134 t^{-2}
- 1494 t^{-1} + 1641 - 1494 t + 1134 t^2 - 717 t^3 + 370 t^4 \\ & - 150 t^5 + 45 t^6
- 9 t^7 + t^8
\end{aligned}  
\ee
\end{tiny}


\noindent $\bullet$ {\bf Symmetries.}
We have $\FKP_K(t) = \FKP_K(t^{-1}) = \FKP_{\overline{K}}(t)$
where $\overline{K}$ is the mirror of $K$.

\noindent $\bullet$ {\bf Mutation.}
Up to 15 crossings, we have $\FKP_K(t) = \FKP_{\mathrm{mut}(K)}(t)$
where $\mathrm{mut}(K)$ is a Conway mutant of $K$. 

\noindent $\bullet$ {\bf Not a Vassiliev power series.}
We have $\FKP_K(1)=1$. However, $\FKP_K$ is not a Vassiliev power series invariant,
i.e., it does not have a weight system~\cite{B-N:Vassiliev}. Indeed, using
the value of the invariant for the $3_1$ and the $4_1$ knots from above
we see that
\be
\FKP_{3_1}(e^\hbar) =1 + 8 \hbar + O(\hbar^2), \qquad
\FKP_{4_1}(e^\hbar) =1 -2 \hbar + O(\hbar^2) \,.
\ee
On the other hand, there is only one degree 2 Vassiliev invariant $V_2$ of knots
taking (up to normalization) the values $V_{2,3_1} = -1$, $V_{2,4_1} = 1$.

\noindent $\bullet$ {\bf Genus bounds.}
For all knots with $\leq 16$ crossings, we have: 
\be
\label{genusb}
\deg_t \FKP_K(t) \leq 8 \cdot \text{genus}(K)
\ee
where by degree of a Laurent polynomial we mean the highest power of $t$ minus
the lowest poer of $t$, and by genus we mean the smallest genus of an orientable
(Seifert) surface that bounds the knot and $\Delta(t) \in \BZ[t^{\pm 1}]$ denotes
the (symmetrized) Alexander polynomial. The inequality~\eqref{genusb} follows from
the work of L\'{o}pez Neumann and van der Veen~\cite{LNV:genus} and upcoming work
of Andruskiewitsch, L\'{o}pez Neumann and S.G., after bosonizing and doubling the
Nichols algebra $\FK_3$, and noting that $\FK_3$ is a graded algebra with
top-degree $4$. 

Thus, genus-wise $\FKP(t)$ behaves like $\Delta(t)^4$. In analogy with the
polynomial invariants from Lie superalgebras~\cite{GL:genus} and with the
Melvin--Morton--Rozansky Conjecture (a theorem in~\cite{BG:MMR}), one would
conjecture that
\be
\label{genusa}
4 \cdot \deg_t \Delta_K(t) \leq \deg_t \FKP_K(t) 
\ee
Alas, this holds for all knots with $\leq 16c$ crossings except for
a single pair of mutant knots
\be
\label{spair}
(16n425245, 16n427445)
\ee
which has $\FKP$-degree 14, Alex-degree 4 and genus 2. Its $\FKP$-polynomial
is
\begin{tiny}
\be
\begin{aligned}  
-264 t^{-7} + 2718 t^{-6} - 13950 t^{-5} + 46927 t^{-4} - 113967 t^{-3}
+ 209455 t^{-2} - 299133 t^{-1} + 336429 & \\ - 299133 t + 209455 t^2 - 113967 t^3
+ 46927 t^4 - 13950 t^5 + 2718 t^6 - 264 t^7 &
\end{aligned}
\ee
\end{tiny}
and the Alexander polynomial is
\begin{tiny}$6 t^{-2} - 29 t^{-1} + 47 - 29 t + 6 t^2$\end{tiny}.

The failure of~\eqref{genusa} is another indication that the $\FKP$ polynomial
is not related to the quantum knot polynomials that come from representation
theory of simple Lie algebras and superalgebras.

Note that the $t$-degree of $\FKP_K(t)$ is always even, but not necessarily
a multiple of 4; see the example the pair of knots in~\eqref{spair} and the knot
$\mathrm{Wh}(6_2)$ below.

\noindent $\bullet$ {\bf Alexander polynomial 1 mutants.}
The mutant knots $11n42$ and $11n34$ (the famous Kinoshita-Terasaka and Conway
pair) have trivial Alexander polynomial but nontrivial common $\FKP$-polynomial given by

\begin{tiny}
\be
\label{11n34}
18 t^{-6} - 150 t^{-5} + 612 t^{-4} - 1614 t^{-3} + 3066 t^{-2} - 4428 t^{-1} + 4993 - 
 4428 t + 3066 t^2 - 1614 t^3 + 612 t^4 - 150 t^5 + 18 t^6 \,.
\ee
\end{tiny}
In particular, the Alexander module (which is trivial, i.e., equal to the unknot
in the above pair of knots), does not determine the $\FKP$-polynomial.

\noindent $\bullet$ {\bf Not determined by Knot Floer Homology.}
The knots $7_2$ and $\overline{9_2}$ have equal Knot Floer Homology but different
$\FKP$-polynomial. 

\noindent $\bullet$ {\bf Not determined by HOMFLY-PT.}
The knots $5_1$ and $10_{132}$ have equal HOMFLY-PT polynomial but different
$\FKP$-polynomial.

\noindent $\bullet$ {\bf Not determined by the $\theta$ invariant.}
The knots $10_{106}$ and $12n369$ have equal $\theta$ (i.e., 2-loop)
invariant~\cite{BNV:theta} but different $\FKP$-polynomial. 

\noindent $\bullet$ {\bf $\FKP$ versus the Alexander polynomial.}
Up to 16 crossings both the Alexander and the $\FKP$ polynomials are
invariant under mirror image, and also under Conway mutation (up to 15 crossings).
Up to 16 crossings, all knots with the same $\FKP$ have the same Alexander polynomial.

\noindent $\bullet$ {\bf Strength of $\FKP$.}
Knots with the same $\FKP$ can have a wild range of number of crossings.
For example, the following knots
\be
8n1, \, 10n29, \, 14n18212, \, 15n41340, \, 15n45749, \, 16n675897
\ee
all have the same $\FKP$, namely
\begin{tiny}
\be
t^{-8} - 6 t^{-7} + 23 t^{-6} - 60 t^{-5} + 109 t^{-4} - 138 t^{-3} + 125 t^{-2}
- 90 t^{-1} 
+ 73 - 90 t + 125 t^2 - 138 t^3 + 109 t^4 - 60 t^5 + 23 t^6 - 6 t^7 + t^8
\ee
\end{tiny}
A complete list of tuples of $\leq 16$ crossing knots with the same values
of $\FKP$ has been included in \cite{GL:FK3data}.

\noindent $\bullet$ {\bf Whitehead doubling.}
If $\mathrm{Wh}(K)$ denotes the Whitehead double of a 0-framed knot $K$ with
positive clasp, we computed the polynomial of the Whitehead double
of all knots with at most 8 crossings and we present a sample of values:

\begin{tiny}
\be
\label{wh}
\begin{aligned}
\FKP_{\mathrm{Wh}(3_1)} &=
-12 t^{-4} + 72 t^{-3} - 228 t^{-2} + 432 t^{-1} - 527 + 432 t - 228 t^2 + 72 t^3
- 12 t^4
\\
\FKP_{\mathrm{Wh}(4_1)} &=
72 t^{-4} - 432 t^{-3} + 1176 t^{-2} - 2016 t^{-1} + 2401 - 2016 t + 1176 t^2 - 432 t^3
+ 72 t^4
\\
\FKP_{\mathrm{Wh}(5_1)} &=
-72 t^{-4} + 504 t^{-3} - 1512 t^{-2} + 2736 t^{-1} - 3311 + 2736 t - 1512 t^2 + 504 t^3
- 72 t^4
\\
\FKP_{\mathrm{Wh}(5_2)} &=
-24 t^{-4} + 72 t^{-3} - 96 t^{-2} + 72 t^{-1} - 47 + 72 t - 96 t^2 + 72 t^3 - 24 t^4
\\
\FKP_{\mathrm{Wh}(6_1)} &=
84 t^{-4} - 432 t^{-3} + 1044 t^{-2} - 1656 t^{-1} + 1921 - 1656 t + 1044 t^2 - 432 t^3
+ 84 t^4
\\
\FKP_{\mathrm{Wh}(6_2)} &=
144 t^{-3} - 624 t^{-2} + 1296 t^{-1} - 1631 + 1296 t - 624 t^2 + 144 t^3
\\
\FKP_{\mathrm{Wh}(6_3)} &=
-264 t^{-4} + 1584 t^{-3} - 4224 t^{-2} + 7128 t^{-1} - 8447 + 7128 t - 4224 t^2
+ 1584 t^3 - 264 t^4
\end{aligned}  
\ee
\end{tiny}


\noindent $\bullet$ {\bf $(2,1)$-cabling.}
If $K(2,1)$ denotes the $(2,1)$-parallel of a 0-framed knot $K$, we 
computed the polynomial of the (2,1)-cable of all knots with at most 8 crossings
and we present a sample of values:

\begin{tiny}
\be
\label{21cab}
\begin{aligned}
\FKP_{3_1(2,1)} &=
t^{-8} + 6 t^{-7} - 21 t^{-6} + 36 t^{-5} - 41 t^{-4} + 30 t^{-3} - 9 t^{-2}
- 12 t^{-1} + 21 - 12 t - 9 t^2 + 30 t^3 - 41 t^4 + 36 t^5 - 21 t^6 \\ &
+ 6 t^7 + t^8
\\
\FKP_{4_1(2,1)} &=
t^{-8} - 9 t^{-6} + 18 t^{-5} - 3 t^{-4} - 54 t^{-3} + 111 t^{-2} - 132 t^{-1} + 137
- 132 t + 111 t^2 - 54 t^3 - 3 t^4 + 18 t^5 - 9 t^6 + t^8
\\
\FKP_{5_1(2,1)} &=
t^{-16} - 3 t^{-14} + 6 t^{-13} - 5 t^{-12} - 12 t^{-10} + 42 t^{-9} - 54 t^{-8}
+ 42 t^{-7} - 33 t^{-6} + 54 t^{-5} - 76 t^{-4} + 60 t^{-3} - 54 t^{-2} \\ & + 84 t^{-1}
- 103 + 84 t - 54 t^2 + 60 t^3 - 76 t^4 + 54 t^5 - 33 t^6 + 42 t^7 - 54 t^8 + 42 t^9
- 12 t^{10} - 5 t^{12} + 6 t^{13} - 3 t^{14} + t^{16}
\\
\FKP_{5_2(2,1)} &=
10 t^{-8} - 12 t^{-7} - 12 t^{-6} + 48 t^{-5} - 36 t^{-4} - 48 t^{-3} + 84 t^{-2}
- 12 t^{-1} - 43 - 12 t + 84 t^2 - 48 t^3 - 36 t^4 + 48 t^5 \\ &
- 12 t^6 - 12 t^7 + 10 t^8
\\
\FKP_{6_1(2,1)} &=
10 t^{-8} - 6 t^{-7} - 90 t^{-6} + 234 t^{-5} - 206 t^{-4} - 84 t^{-3} + 306 t^{-2}
- 228 t^{-1} + 129 - 228 t + 306 t^2 - 84 t^3 - 206 t^4 \\ &
+ 234 t^5 - 90 t^6 - 6 t^7 + 10 t^8
\\
\FKP_{6_2(2,1)} &=
t^{-16} - 9 t^{-14} + 18 t^{-13} + 3 t^{-12} - 48 t^{-11} + 12 t^{-10} + 108 t^{-9}
- 102 t^{-8} - 96 t^{-7} + 159 t^{-6} + 102 t^{-5} - 264 t^{-4} \\ &
+ 24 t^{-3} + 168 t^{-2} + 84 t^{-1} - 319 + 84 t + 168 t^2 + 24 t^3 - 264 t^4
+ 102 t^5 + 159 t^6 - 96 t^7 - 102 t^8 + 108 t^9 + 12 t^{10} \\ &
- 48 t^{11} + 3 t^{12} + 18 t^{13} - 9 t^{14} + t^{16}
\\
\FKP_{6_3(2,1)} &=
t^{-16} - 9 t^{-14} + 18 t^{-13} + 9 t^{-12} - 66 t^{-11} + 30 t^{-10} + 126 t^{-9}
- 122 t^{-8} - 126 t^{-7} + 99 t^{-6} + 408 t^{-5} - 570 t^{-4} \\ &
- 84 t^{-3} + 450 t^{-2} + 252 t^{-1} - 831 + 252 t + 450 t^2 - 84 t^3 - 570 t^4
+ 408 t^5 + 99 t^6 - 126 t^7 - 122 t^8 + 126 t^9 + 30 t^{10} \\ &
- 66 t^{11} + 9 t^{12} + 18 t^{13} - 9 t^{14} + t^{16}
\end{aligned}  
\ee
\end{tiny}



\section{Further directions}
\label{sec.further}

\subsection{Summary}
\label{sub.summary}

Let us summarize our discussion of the $\FK_3$-knot polynomial. Every solution
$(V,\tau)$ to the Yang--Baxter equation on a finite dimensional vector space $V$ with
$\tau \in \End(V\otimes V)$ defines a braided vector space and hence determines a
Nichols algebra $\calB(V,\tau)$. If we start with the solution 
~\eqref{braidingmat} of the Yang--Baxter equation on a 3-dimensional vector
space, (based on the $S_3$-YD-module),
we can define the Nichols algebra $\calB(V,\tau)$ and use its automorphism to
define the $R$-matrix on the 12-dimensional space $\calB(V,\tau)$ and the corresponding
knot invariant. Schematically, we have:
\be
([(12)],\text{sgn}) \rightsquigarrow \tau \in \End(V^{\otimes 2})
\rightsquigarrow \calB(V,\tau)
\stackrel{+\text{auto}}{\rightsquigarrow} R \in
\End(\calB(V,\tau)^{\otimes 2})  \stackrel{\mathrm{RT}}{\rightsquigarrow}
\text{Knot Polynomial}
\ee
It is rather remarkable that starting from such an elementary quandle of $S_3$, we
produced such a nontrivial knot polynomial.

It is reasonable to expect that the value of the $\FK_3$-knot polynomial at $t=1$
is related to the values of the knot invariants of the 8 simple
$S_3$-YD-modules discussed in Lemma~\ref{lem.S3}, and that the $t=e^\hbar$-expansion
of the $\FK_3$-knot polynomial to be some nontrivial deformation of the
number of Fox 3-colorings of $K$. This would also explain why the
$\FK_3$-polynomial is not a Vassiliev power series invariant.

\begin{question}
\label{que.def}
Even though the $\FK_3$ Nichols algebra has no deformations (within the category
of Nichols algebras), does its $R$-matrix have any non-trivial deformations?
\end{question}

A straightforward but perhaps impractical approach to find such deformations would be
to search for all solutions
$R(I+\ve M) + O(\ve^2)$ to the Yang--Baxter equation modulo $O(\ve^2)$, where $M$ is
an unknown $12^2 \times 12^2$ matrix. This gives a sparse system of
$12^6=\num{2985984}$ linear equations with $12^4=\num{20736}$ unknowns and with
coefficients in the field $\BQ(t)$. For all such solutions, one can then
compute the associated polynomial of a few knots and see whether it depends
(up to an overall rescaling of $\ve$) to some polynomials in the entries of $M$.
If that happens, then we obtain a nontrivial $\ve$-deformed knot polynomial. We
thank Rinat Kashaev for suggesting this problem to us.

\subsection{Beyond 12-dimensional Nichols algebras}
\label{sec.beyond}

Every finite dimensional Nichols algebra $H$ with automorphism gives a solution
to the Yang--Baxter equation~\cite{GK:multi}. Thus, we may search beyond 12-dimensional
solutions to the Yang--Baxter equation coming from finite-dimensional Nichols
algebras. In the small list of finite dimensional elementary Nichols algebras (see
~\cite[Tab.9.1]{Heckenberger:cubic} and also ~\cite[Tab.1]{Grana}),
the next Nichols algebra after $\FK_3$ in increasing dimension is one of dimension
36 and rank 4 over the finite field with 2 elements, and then one of dimension 72
and rank 4 over $\BQ$, which is also sporadic. It will be interesting to compute their
$R$ matrices and the corresponding knot polynomials for a few knots.

\subsection*{Acknowledgements}

We wish to thank Nicolas Andruskiewitsch and Rinat Kashaev for many
enlightening conversations.

\appendix

\section{The 12-dimensional $R$-matrix}
\label{appen.R}

The condensed $R$-matrix $(R(b_i \otimes b_j))_{i,j=0,\dots,11}$ is a $12 \times 12$
matrix given with the convention that $b_{i,j}=b_i \otimes b_j$.
Due to the size of the matrix, we partition each row of length 12 
into 4 vectors of size 3. As we progress down the rows, the size of their
entries increases, with the last row containing the biggest entry $R(b_{11,0})$.
With this explanation, the rows are given by:

\begin{tiny}
\begin{align*}
R(b_{0,0}) &=
b_{0,0}, \quad R(b_{0,1}) =t b_{1,0}, \quad R(b_{0,2}) =t b_{2,0} \\
R(b_{0,3}) &=t b_{3,0}, \quad R(b_{0,4}) =t^2 b_{4,0}, \quad R(b_{0,5}) =t^2 b_{5,0} \\
R(b_{0,6}) &=t^2 b_{6,0}, \quad R(b_{0,7}) =t^2 b_{7,0}, \quad R(b_{0,8}) =t^3 b_{8,0} \\
R(b_{0,9}) &=t^3 b_{9,0}, \quad R(b_{0,10}) =t^3 b_{10,0}, \quad R(b_{0,11}) =t^4 b_{11,0} \\
\end{align*}
\end{tiny}
\begin{tiny}
\begin{align*}
R(b_{1,0}) &=b_{0,1}+(1-t) b_{1,0}, \quad R(b_{1,1}) =-t b_{1,1}, \quad R(b_{1,2}) = \left(t-t^2\right) b_{4,0}-t b_{3,1} \\
R(b_{1,3}) &=\left(t-t^2\right) b_{6,0}-t b_{2,1}, \quad R(b_{1,4}) =t^2 b_{6,1}, \quad R(b_{1,5}) =-b_{7,1} t^2-t^2 b_{4,1}+\left(t^2-t^3\right) b_{9,0} \\
R(b_{1,6}) &=t^2 b_{4,1}, \quad R(b_{1,7}) =-b_{6,1} t^2-t^2 b_{5,1}+\left(t^2-t^3\right) b_{8,0}, \quad R(b_{1,8}) =t^3 b_{9,1} \\
R(b_{1,9}) &=t^3 b_{8,1}, \quad R(b_{1,10}) =\left(t^3-t^4\right) b_{11,0}-t^3 b_{10,1}, \quad  R(b_{1,11}) =t^4 b_{11,1} 
\end{align*}
\end{tiny}

\begin{tiny}
\begin{align*}
R(b_{2,0}) &=b_{0,2}+(1-t) b_{2,0}, \quad R(b_{2,1}) =\left(t-t^2\right) b_{5,0}-t b_{3,2}, \quad R(b_{2,2}) =-t b_{2,2} \\
R(b_{2,3}) &=\left(t-t^2\right) b_{7,0}-t b_{1,2}, \quad R(b_{2,4}) =-b_{6,2} t^2-t^2 b_{5,2}+\left(t^2-t^3\right) b_{9,0}, \quad R(b_{2,5}) =t^2 b_{7,2} \\
R(b_{2,6}) &=-b_{7,2} t^2-t^2 b_{4,2}+\left(t^2-t^3\right) b_{10,0}, \quad R(b_{2,7}) =t^2 b_{5,2}, \quad R(b_{2,8}) =\left(t^3-t^4\right) b_{11,0}-t^3 b_{8,2}
\\ R(b_{2,9}) &=t^3 b_{10,2}, \quad R(b_{2,10}) =t^3 b_{9,2}, \quad R(b_{2,11}) =t^4 b_{11,2}
\end{align*}
\end{tiny}

\begin{tiny}
\begin{align*}
R(b_{3,0}) &=b_{0,3}+(1-t) b_{3,0}, \quad R(b_{3,1}) =-t b_{2,3}+\left(t^2-t\right) b_{4,0}+\left(t^2-t\right) b_{7,0}, \quad R(b_{3,2}) =-t b_{1,3}+\left(t^2-t\right) b_{5,0}+\left(t^2-t\right) b_{6,0} \\
R(b_{3,3}) &=-t b_{3,3}, \quad R(b_{3,5}) =b_{5,4} t^2+\left(t^2-t^3\right) b_{10,0}, \quad R(b_{3,5}) =b_{4,3} t^2+\left(t^2-t^3\right) b_{8,0} \\
R(b_{3,6}) &=b_{7,3} t^2+\left(t^3-t^2\right) b_{8,0}, \quad R(b_{3,8}) =b_{6,3} t^2+\left(t^3-t^2\right) b_{10,0}, \quad R(b_{3,8}) =-t^3 b_{10,3} \\
R(b_{3,9}) &=\left(t^4-t^3\right) b_{11,0}-t^3 b_{9,3}, \quad R(b_{3,10}) =-t^3 b_{8,3}, \quad R(b_{3,11}) =t^4 b_{11,3}
\end{align*}
\end{tiny}

\begin{tiny}
\begin{align*}
R(b_{4,0}) &=b_{0,4}+b_{1,2}-t b_{2,3}+(t-1) b_{3,1}+\left(t^2-2 t+1\right) b_{4,0}+\left(t^2-t\right) b_{7,0} \\
R(b_{4,1}) &=t b_{2,4}+\left(t^2-t\right) b_{4,1}-t b_{6,2}+\left(t^2-t\right) b_{7,1}+\left(t-t^2\right) b_{9,0}, \quad R(b_{4,2}) =b_{5,3} t^2+b_{3,4} t-b_{4,2} t+\left(t^2-t^3\right) b_{10,0} \\
R(b_{4,3}) &=b_{7,3} t^2+b_{1,4} t+\left(t^2-t\right) b_{5,1}+\left(t^2-t\right) b_{6,1}+\left(t^3-2 t^2+t\right) b_{8,0}, \quad R(b_{4,4}) =b_{7,4} t^2-b_{9,2} t^2+\left(t^2-t^3\right) b_{8,1} \\
R(b_{4,5}) &=-b_{9,3} t^3-b_{6,4} t^2+b_{8,2} t^2-t^2 b_{5,4}+\left(t^4-t^3\right) b_{11,0}, \quad R(b_{4,6}) =b_{5,4} t^2-b_{8,2} t^2+\left(t^3-t^2\right) b_{10,1}+\left(t^2-t^3\right) b_{11,0} \\
R(b_{4,7}) &=-b_{10,3} t^3-b_{7,4} t^2+b_{9,2} t^2-t^2 b_{4,4}, \quad R(b_{4,8}) =\left(t^3-t^4\right) b_{11,1}-t^3 b_{9,4} \\
R(b_{4,9}) &=t^3 b_{11,2}-t^3 b_{10,4}, \quad R(b_{4,10}) =b_{11,3} t^4+b_{8,4} t^3, \quad R(b_{4,11}) =t^4 b_{11,4} 
\end{align*}
\end{tiny}

\begin{tiny}
\begin{align*}
R(b_{5,0}) &=b_{0,5}-t b_{1,3}+b_{2,1}+(t-1) b_{3,2}+\left(t^2-2 t+1\right) b_{5,0}+\left(t^2-t\right) b_{6,0}, \quad R(b_{5,1}) =b_{4,3} t^2+b_{3,5} t-b_{5,1} t+\left(t^2-t^3\right) b_{8,0} \\
R(b_{5,2}) &=t b_{1,5}+\left(t^2-t\right) b_{5,2}+\left(t^2-t\right) b_{6,2}-t b_{7,1}+\left(t-t^2\right) b_{9,0} \\
R(b_{5,3}) &=b_{6,3} t^2+b_{2,5} t+\left(t^2-t\right) b_{4,2}+\left(t^2-t\right) b_{7,2}+\left(t^3-2 t^2+t\right) b_{10,0} \\
R(b_{5,4}) &=-b_{9,3} t^3-b_{7,5} t^2+b_{10,1} t^2-t^2 b_{4,5}+\left(t^4-t^3\right) b_{11,0}, \quad R(b_{5,5}) =b_{6,5} t^2-b_{9,1} t^2+\left(t^2-t^3\right) b_{10,2} \\
R(b_{5,6}) &=-b_{8,3} t^3-b_{6,5} t^2+b_{9,1} t^2-t^2 b_{5,5}, \quad R(b_{5,7}) =b_{4,5} t^2-b_{10,1} t^2+\left(t^3-t^2\right) b_{8,2}+\left(t^2-t^3\right) b_{11,0} \\
R(b_{5,8}) &=b_{11,3} t^4+b_{10,5} t^3, \quad R(b_{5,9}) =t^3 b_{11,1}-t^3 b_{8,5}, \quad R(b_{5,10}) =\left(t^3-t^4\right) b_{11,2}-t^3 b_{9,5}, \quad R(b_{5,11}) =t^4 b_{11,5} 
\end{align*}
\end{tiny}

\begin{tiny}
\begin{align*}
R(b_{6,0}) &=b_{0,6}+b_{1,3}+(t-1) b_{2,1}-t b_{3,2}+\left(t-t^2\right) b_{5,0}+(1-t) b_{6,0}, \quad R(b_{6,1}) =t b_{3,6}-t b_{4,3}+\left(t-t^2\right) b_{5,1}+\left(t^2-t\right) b_{8,0} \\
R(b_{6,2}) &=-b_{5,2} t^2-b_{6,2} t^2+b_{1,6} t+\left(t-t^2\right) b_{7,1}+\left(-t^3+2 t^2-t\right) b_{9,0}, \quad
R(b_{6,3}) =-b_{4,2} t^2-b_{7,2} t^2+b_{2,6} t-b_{6,3} t+\left(t^2-t^3\right) b_{10,0} \\
R(b_{6,4}) &=-b_{7,6} t^2+b_{9,3} t^2-t^2 b_{4,6}+\left(t^3-t^2\right) b_{10,1}+\left(t^2-t^3\right) b_{11,0}, \quad R(b_{6,5}) =b_{10,2} t^3+b_{6,6} t^2+\left(t^2-t^3\right) b_{9,1} \\
R(b_{6,6}) &=-b_{6,6} t^2+b_{8,3} t^2-t^2 b_{5,6}+\left(t^3-t^2\right) b_{9,1}, \quad
R(b_{6,7}) =-b_{8,2} t^3+b_{4,6} t^2+\left(t^2-t^3\right) b_{10,1}+\left(-t^4+2 t^3-t^2\right) b_{11,0} \\
R(b_{6,8}) &=t^3 b_{10,6}-t^3 b_{11,3}, \quad R(b_{6,9}) =\left(t^4-t^3\right) b_{11,1}-t^3 b_{8,6}, \quad R(b_{6,10}) =t^4 b_{11,2}-t^3 b_{9,6}, \quad R(b_{6,11}) =t^4 b_{11,6}
\end{align*}
\end{tiny}

\begin{tiny}
\begin{align*}
R(b_{7,0}) &=b_{0,7}+(t-1) b_{1,2}+b_{2,3}-t b_{3,1}+\left(t-t^2\right) b_{4,0}+(1-t) b_{7,0} \\
R(b_{7,1}) &=-b_{4,1} t^2-b_{7,1} t^2+b_{2,7} t+\left(t-t^2\right) b_{6,2}+\left(-t^3+2 t^2-t\right) b_{9,0} \\
R(b_{7,2}) &=t b_{3,7}+\left(t-t^2\right) b_{4,2}-t b_{5,3}+\left(t^2-t\right) b_{10,0}, \quad R(b_{7,3}) =-b_{5,1} t^2-b_{6,1} t^2+b_{1,7} t-b_{7,3} t+\left(t^2-t^3\right) b_{8,0} \\
R(b_{7,4}) &=b_{8,1} t^3+b_{7,7} t^2+\left(t^2-t^3\right) b_{9,2}, \quad R(b_{7,5}) =-b_{6,7} t^2+b_{9,3} t^2-t^2 b_{5,7}+\left(t^3-t^2\right) b_{8,2}+\left(t^2-t^3\right) b_{11,0} \\
R(b_{7,6}) &=-b_{10,1} t^3+b_{5,7} t^2+\left(t^2-t^3\right) b_{8,2}+\left(-t^4+2 t^3-t^2\right) b_{11,0}, \quad R(b_{7,7}) =-b_{7,7} t^2+b_{10,3} t^2-t^2 b_{4,7}+\left(t^3-t^2\right) b_{9,2} \\
R(b_{7,8}) &=t^4 b_{11,1}-t^3 b_{9,7}, \quad R(b_{7,9}) =\left(t^4-t^3\right) b_{11,2}-t^3 b_{10,7}, \quad R(b_{7,10}) =t^3 b_{8,7}-t^3 b_{11,3}, \quad R(b_{7,11}) =t^4 b_{11,7} 
\end{align*}
\end{tiny}

\begin{tiny}
\begin{align*}
R(b_{8,0}) &=b_{0,8}+(1-t) b_{1,4}+b_{1,7}+t b_{3,5}+(t-1) b_{3,6}+\left(t^2-t+1\right) b_{4,3}+\left(-t^2+t-1\right) b_{5,1}+\left(-t^2+t-1\right) b_{6,1} \\ & +\left(-t^3+2 t^2-2 t+1\right) b_{8,0} \\
R(b_{8,1}) &=-t b_{3,8}+\left(t-t^2\right) b_{4,4}+t b_{4,7}+\left(-t^3+t^2-t\right) b_{8,1} \\
R(b_{8,2}) &=-b_{4,5} t^2-b_{7,5} t^2+b_{6,7} t-t b_{2,8}+\left(t-t^2\right) b_{4,6}+\left(t-t^2\right) b_{6,4}+\left(t-t^2\right) b_{7,6}+\left(-t^3+t^2-t\right) b_{9,3} \\ & +\left(t^3-t^2+t\right) b_{10,1}+\left(t^4-2 t^3+2 t^2-t\right) b_{11,0} \\
R(b_{8,3}) &=-b_{5,5} t^2-b_{6,5} t^2-t b_{1,8}+\left(t-t^2\right) b_{5,6}+\left(t-t^2\right) b_{6,6}+\left(-t^3+t^2-t\right) b_{8,3} \\
R(b_{8,4}) &=-b_{6,8} t^2+b_{8,7} t^2-t^2 b_{5,8}+\left(t^2-t^3\right) b_{8,4} \\
R(b_{8,5}) &=-b_{8,5} t^3+b_{7,8} t^2-b_{9,7} t^2+\left(t^2-t^3\right) b_{8,6}+\left(t^3-t^2\right) b_{9,4}+\left(t^4-t^3+t^2\right) b_{11,1} \\
R(b_{8,6}) &=-b_{7,8} t^2+b_{9,7} t^2-t^2 b_{4,8}+\left(t^2-t^3\right) b_{9,4}+\left(-t^4+t^3-t^2\right) b_{11,1} \\
R(b_{8,7}) &=b_{10,5} t^3+b_{5,8} t^2-b_{8,7} t^2+\left(t^3-t^2\right) b_{8,4}+\left(t^3-t^2\right) b_{10,6}+\left(t^4-t^3+t^2\right) b_{11,3} \\
R(b_{8,8}) &=-t^3 b_{8,8}, \quad R(b_{8,9}) =b_{10,8} t^3-b_{11,7} t^3+\left(t^4-t^3\right) b_{11,4}, \quad R(b_{8,10}) =b_{11,5} t^4+b_{9,8} t^3+\left(t^4-t^3\right) b_{11,6}, \quad R(b_{8,11}) =t^4 b_{11,8}
\end{align*}
\end{tiny}

\begin{tiny}
\begin{align*}
R(b_{9,0}) &=b_{0,9}+b_{1,5}+t b_{1,6}+b_{2,4}+t b_{2,7}+\left(-t^2+t-1\right) b_{6,2}+\left(-t^2+t-1\right) b_{7,1}+\left(-t^3+2 t^2-2 t+1\right) b_{9,0} \\
R(b_{9,1}) &=b_{6,6} t^2+b_{6,5} t-t b_{2,9}+\left(-t^3+t^2-t\right) b_{9,1}, \quad R(b_{9,2}) =b_{7,7} t^2+b_{7,4} t-t b_{1,9}+\left(-t^3+t^2-t\right) b_{9,2} \\
R(b_{9,3}) &=b_{4,6} t^2+b_{5,7} t^2+b_{4,5} t+b_{5,4} t-t b_{3,9}+\left(-t^3+t^2-t\right) b_{8,2}+\left(-t^3+t^2-t\right) b_{10,1}+\left(-t^4+2 t^3-2 t^2+t\right) b_{11,0} \\
R(b_{9,4}) &=-b_{8,6} t^3+b_{5,9} t^2-b_{8,5} t^2+\left(t^4-t^3+t^2\right) b_{11,1}, \quad R(b_{9,5}) =-b_{10,7} t^3+b_{4,9} t^2-b_{10,4} t^2+\left(t^4-t^3+t^2\right) b_{11,2} \\
R(b_{9,6}) &=-b_{9,6} t^3+b_{7,9} t^2-b_{9,5} t^2, \quad R(b_{9,7}) =-b_{9,7} t^3+b_{6,9} t^2-b_{9,4} t^2, \quad R(b_{9,8}) =b_{11,6} t^4+b_{11,5} t^3-t^3 b_{10,9} \\
R(b_{9,9}) &=-t^3 b_{9,9}, \quad  R(b_{9,10}) =b_{11,7} t^4+b_{11,4} t^3-t^3 b_{8,9}, \quad R(b_{9,11}) =t^4 b_{11,9} 
\end{align*}
\end{tiny}

\begin{tiny}
\begin{align*}
R(b_{10,0}) &=b_{0,10}+(1-t) b_{2,5}+b_{2,6}+t b_{3,4}+(t-1) b_{3,7}+\left(-t^2+t-1\right) b_{4,2}+\left(t^2-t+1\right) b_{5,3}+\left(-t^2+t-1\right) b_{7,2} \\ & +\left(-t^3+2 t^2-2 t+1\right) b_{10,0} \\
R(b_{10,1}) &=-b_{5,4} t^2-b_{6,4} t^2+b_{7,6} t-t b_{1,10}+\left(t-t^2\right) b_{5,7}+\left(t-t^2\right) b_{6,7}+\left(t-t^2\right) b_{7,5}+\left(t^3-t^2+t\right) b_{8,2}\\ & +\left(-t^3+t^2-t\right) b_{9,3}+\left(t^4-2 t^3+2 t^2-t\right) b_{11,0} \\
R(b_{10,2}) &=-t b_{3,10}+\left(t-t^2\right) b_{5,5}+t b_{5,6}+\left(-t^3+t^2-t\right) b_{10,2} \\
R(b_{10,3}) &=-b_{4,4} t^2-b_{7,4} t^2-t b_{2,10}+\left(t-t^2\right) b_{4,7}+\left(t-t^2\right) b_{7,7}+\left(-t^3+t^2-t\right) b_{10,3} \\
R(b_{10,4}) &=-b_{10,4} t^3+b_{6,10} t^2-b_{9,6} t^2+\left(t^3-t^2\right) b_{9,5}+\left(t^2-t^3\right) b_{10,7}+\left(t^4-t^3+t^2\right) b_{11,2} \\
R(b_{10,5}) &=-b_{7,10} t^2+b_{10,6} t^2-t^2 b_{4,10}+\left(t^2-t^3\right) b_{10,5} \\
R(b_{10,6}) &=b_{8,4} t^3+b_{4,10} t^2-b_{10,6} t^2+\left(t^3-t^2\right) b_{8,7}+\left(t^3-t^2\right) b_{10,5}+\left(t^4-t^3+t^2\right) b_{11,3} \\
R(b_{10,7}) &=-b_{6,10} t^2+b_{9,6} t^2-t^2 b_{5,10}+\left(t^2-t^3\right) b_{9,5}+\left(-t^4+t^3-t^2\right) b_{11,2}, \quad R(b_{10,8}) =b_{11,4} t^4+b_{9,10} t^3+\left(t^4-t^3\right) b_{11,7} \\
R(b_{10,9}) &=b_{8,10} t^3-b_{11,6} t^3+\left(t^4-t^3\right) b_{11,5}, \quad R(b_{10,10}) =-t^3 b_{10,10}, \quad R(b_{10,11}) =t^4 b_{11,10} 
\end{align*}
\end{tiny}

\begin{tiny}
\begin{align*}
R(b_{11,0}) &=b_{0,11}+(1-t) b_{1,10}+(1-t) b_{2,8}+(t-1) b_{3,9}+(1-t) b_{4,5}+\left(t-t^2\right) b_{4,6}+(1-t) b_{5,4}+\left(t-t^2\right) b_{5,7} \\ & +\left(t-t^2\right) b_{6,4}+\left(-t^2+2 t-1\right) b_{6,7}+\left(t-t^2\right) b_{7,5}+\left(-t^2+2 t-1\right) b_{7,6}+\left(t^3-2 t^2+2 t-1\right) b_{8,2} \\ & +\left(-t^3+2 t^2-2 t+1\right) b_{9,3}+\left(t^3-2 t^2+2 t-1\right) b_{10,1}+\left(t^4-3 t^3+4 t^2-3 t+1\right) b_{11,0} \\
R(b_{11,1}) &=t b_{1,11}+\left(t^2-t\right) b_{5,9}+\left(t^2-t\right) b_{6,9}+\left(t^2-t\right) b_{7,8}+\left(t-t^2\right) b_{8,5}+\left(t^2-t^3\right) b_{8,6}+\left(t^3-2 t^2+t\right) b_{9,4} \\ & +\left(t-t^2\right) b_{9,7}+\left(t^4-2 t^3+2 t^2-t\right) b_{11,1} \\
R(b_{11,2}) &=t b_{2,11}+\left(t^2-t\right) b_{4,9}+\left(t^2-t\right) b_{6,10}+\left(t^2-t\right) b_{7,9}+\left(t^3-2 t^2+t\right) b_{9,5}+\left(t-t^2\right) b_{9,6}+\left(t-t^2\right) b_{10,4} \\ & +\left(t^2-t^3\right) b_{10,7}+\left(t^4-2 t^3+2 t^2-t\right) b_{11,2} \\
R(b_{11,3}) &=t b_{3,11}+\left(t^2-t\right) b_{4,10}+\left(t^2-t\right) b_{5,8}+\left(t^3-t^2\right) b_{8,4}+\left(t^3-2 t^2+t\right) b_{8,7}+\left(t^3-t^2\right) b_{10,5}+\left(t^3-2 t^2+t\right) b_{10,6} \\ & +\left(t^4-2 t^3+2 t^2-t\right) b_{11,3} \\
R(b_{11,4}) &=b_{4,11} t^2+\left(t^3-t^2\right) b_{8,9}+\left(t^3-t^2\right) b_{10,8}+\left(t^4-2 t^3+t^2\right) b_{11,4}+\left(t^2-t^3\right) b_{11,7} \\
R(b_{11,5}) &=b_{5,11} t^2+\left(t^3-t^2\right) b_{8,10}+\left(t^3-t^2\right) b_{10,9}+\left(t^4-2 t^3+t^2\right) b_{11,5}+\left(t^2-t^3\right) b_{11,6} \\
R(b_{11,6}) &=b_{6,11} t^2+\left(t^3-t^2\right) b_{9,8}+\left(t^2-t^3\right) b_{10,9}+\left(t^3-t^2\right) b_{11,5}+\left(t^4-t^3\right) b_{11,6} \\
R(b_{11,7}) &=b_{7,11} t^2+\left(t^2-t^3\right) b_{8,9}+\left(t^3-t^2\right) b_{9,10}+\left(t^3-t^2\right) b_{11,4}+\left(t^4-t^3\right) b_{11,7}, \quad
R(b_{11,8}) =b_{8,11} t^3+\left(t^4-t^3\right) b_{11,8} \\
R(b_{11,9}) &=b_{9,11} t^3+\left(t^4-t^3\right) b_{11,9}, \quad R(b_{11,10}) =b_{10,11} t^3+\left(t^4-t^3\right) b_{11,10}, \quad R(b_{11,11}) =t^4 b_{11,11}
\end{align*}
\end{tiny}


\bibliographystyle{hamsalpha}
\bibliography{biblio}
\end{document}